\documentclass[11pt,reqno]{amsart}

\usepackage{hyperref}
\usepackage{graphicx}
\usepackage{color}
\usepackage{amsfonts,amssymb}
\usepackage{bbm} 
\usepackage{psfrag}
\usepackage{mathrsfs}
\usepackage{pdfsync}
\usepackage{a4wide}

\usepackage{mathabx}

\newtheorem{neu}{}[section]

\newtheorem*{Cor*}{Corollary}
\newtheorem{Thm}[neu]{Theorem}
\newtheorem*{Thm*}{Theorem}

\newtheorem*{Prop*}{Proposition}
\theoremstyle{definition}
\newtheorem{Lemma}[neu]{Lemma}
\newtheorem*{Rmk*}{Remark}
\newtheorem{Rmk}[neu]{Remark}

\newtheorem*{Ex*}{Example}

\newtheorem*{Qu*}{Question}

\newtheorem{Conv}[neu]{Convention}
\newtheorem{Assump}{Assumption}

\newcommand{\N}{\mathbb{N}}

\newcommand{\Z}{\mathbb{Z}}
\newcommand{\R}{\mathbb{R}}
\newcommand{\C}{\mathbb{C}}

\newcommand{\om}{\omega}

\newcommand{\ev}{\mathrm{ev}}

\newcommand{\A}{\mathcal{A}}

\renewcommand{\S}{\mathfrak{S}}

\newcommand{\F}{\mathcal{F}}

\newcommand{\M}{\mathcal{M}}

\renewcommand{\H}{\mathrm{H}}

\newcommand{\CM}{\mathrm{CM}}
\newcommand{\HM}{\mathrm{HM}}

\newcommand{\Crit}{\mathrm{Crit}}

\newcommand{\beq}{\begin{equation}}
\newcommand{\beqn}{\begin{equation}\nonumber}
\newcommand{\eeq}{\end{equation}}

\newcommand{\bea}{\begin{equation}\begin{aligned}}
\newcommand{\bean}{\begin{equation}\begin{aligned}\nonumber}
\newcommand{\eea}{\end{aligned}\end{equation}}

\usepackage{marginnote}

\numberwithin{equation}{section}

\definecolor{Urs}{rgb}{0,.7,0}
\definecolor{Peter}{rgb}{0,0,1}
\definecolor{red}{rgb}{1,0,0}

\newcommand{\p}{\partial}

\newcommand{\Hess}{\mathrm{Hess}}

\begin{document}
\title{Cuplength estimates in Morse cohomology}
\author{Peter Albers}
\author{Doris Hein}
\address{
    Peter Albers\\
   Mathematisches Institut\\
    Westf\"alische Wilhelms-Universit\"at M\"unster}
\email{peter.albers@wwu.de}
\address{Doris Hein\\
Mathematisches Institut\\
Albert-Ludwigs-Universit\"at Freiburg}
\email{doris.hein@math.uni-freiburg.de}
\subjclass[2000]{}
\begin{abstract}
The main goal of this paper is to give a unified treatment to many known cuplength estimates with a view towards Floer theory.
As the base case, we prove that for $C^0$-perturbations of a function which is Morse-Bott along a closed submanifold, the number of critical points is bounded below in terms of the cuplength of that critical submanifold. As we work with rather general assumptions the proof also applies in a variety of Floer settings. For example, this proves lower bounds (some of which were known) for the number of fixed points of Hamiltonian diffeomorphisms, Hamiltonian chords for Lagrangian submanifolds, translated points of contactomorphisms, and solutions to a Dirac-type equation.
\end{abstract}
\maketitle

\section{Introduction}

In this paper, we prove a rather general cuplength estimate in Morse homology and explain how this generalizes in the context of Floer theory. The main goal is to give a unified treatment of many known cuplength-type results in Floer theory including one improvement and a new result.

Let $M$ be a manifold (not necessarily compact) and $F\colon M\to\R$ be a smooth function. We assume that $Z\subset M$ is a closed (i.e.~compact and without boundary), connected submanifold satisfying $Z\subset\Crit F$ and $F$ is Morse-Bott along $Z$, that is, for all $z\in Z$ we have
\beq
\ker \Hess_z F=T_zZ\;.
\eeq
For convenience of notation, we also assume that $F|_Z=0$. Furthermore, we assume that the spectral gap $\S$ of $F$ with respect to $Z$ is positive.
Here, the spectral gap is defined by
\beq
\S\equiv\S(F,Z):=\inf\{|F(x)|\colon x\in\Crit F\setminus Z\}>0\;.
\eeq
If $F$ does not have critical points outside of $Z$, we set $\S:=\infty$. Since $M$ is not assumed to be compact, we have to make an assumption which guarantees that solution spaces to (generalized) gradient flow equations are well-behaved. This is Assumption \ref{Ass} on $C^\infty_{loc}$-compactness below. 

\begin{Thm}
\label{thm:main_intro}
Let $h:M\to\R$ be a smooth function such that
\beq
\|h-F\|:=\sup(h-F)-\inf(h-F)<\S
\eeq
and Assumption \ref{Ass} is satisfied, then $h$ has at least $\mathrm{cuplength}(Z)+1$ critical points with critical values in the interval $[-\|h-F\|,\|h-F\|]$.
\end{Thm}

We recall that the cuplength of the critical submanifold $Z$ is defined as
\beq
\mathrm{cuplength}(Z):=\max\{k\in\N\mid \exists a_1,\ldots,a_k\in\H^{\geq1}(Z)\text{ such that } a_1\cup\ldots\cup a_k\neq0\}
\eeq
where $\H^{\geq1}(Z)$ denotes the cohomology in degree at least 1. For simplicity we use in this article cohomology with $\Z/2$-coefficients. We point out that in order to realize the cup-product on the cohomology of $Z$ in this paper we only use finite dimensional Morse cohomology of $Z$ and never Floer homology.

For a more detailed discussion of Assumption \ref{Ass}, more technical definitions are needed, see Section \ref{sec:Setting} for the precise statement and Section \ref{compactness} for how it implies compactness of all relevant moduli spaces.
Assumption \ref{Ass} holds if $M$ is compact or, more generally, if the functions involved are coercive. 

\begin{Rmk}
Theorem \ref{thm:main_intro} has probably proofs which are closer to classical Morse theory. The proof we present here is with an eye towards Floer-theoretic generalizations mentioned below. Indeed, the following stronger version of Theorem \ref{thm:main_intro} should be true. If the number of distinct critical values of $h$ is strictly smaller than $\mathrm{cuplength}(Z) +1$ then the cohomology of the critical set of $h$ is non-trivial. A Floer theoretic version of this has been proved in \cite{How12}.
\end{Rmk}

\subsection*{Acknowledgments} 
We are grateful to Urs Frauenfelder for discussions on an earlier version of this article. We thank the anonymous referee for his/her very thorough work. PA is supported by the SFB 878 - Groups, Geometry and Actions. Some of the work was carried out while DH was at the Institute for Advanced Study and supported by the NSF grant DMS-1128155.

\section{Floer theoretic applications}

We now illustrate Theorem \ref{thm:main_intro} on Floer theoretic examples. Even though Theorem \ref{thm:main_intro} is stated for finite-dimensional Morse theory we explain in detail in Section \ref{sec:Floer} how the proof of Theorem \ref{thm:main_intro} has to be adjusted in the context of Floer theory. Even though Theorem \ref{thm:main_intro} is of independent interest, our main focus is giving a unified explanation and proof of several cuplength estimates in Floer theory. Theorem \ref{thm:main_intro} is applied via the analogy of Floer homology as a half-infinite dimensional Morse theory on the loop space of a symplectic manifold. To ensure that Assumption \ref{Ass} holds in these cases, we now let $(W,\om)$ be a closed symplectic manifold.

\subsection{Symplectically aspherical manifolds}

First we assume that $(W,\om)$ is symplectically aspherical, that is $\om|_{\pi_2(M)}=0$. Then the symplectic area functional $\A:\Lambda W\to\R$ is defined on the space $\Lambda W:=C^\infty_{contr}(S^1,W)$ of contractible loops by
\beq
\A(x):=-\int_{D^2}u^*\om
\eeq
where the capping $u$ of $x$ is a smooth map $u:D^2\to W$ with $u|_{S^1}=x$. In this situation, we set $F:=\A$ and $M:=\Lambda W$ and observe that $\Crit F$ is the set constant loops and can be identified with $W$. We set
\beq
Z:=\Crit F=W\;.
\eeq
It follows from the non-degeneracy of the symplectic form that $F$ is indeed Morse-Bott along $Z$ and, since $F$ has no other critical points, we have $\S=\infty$. 

The function $h$ is then defined using Hamiltonian perturbations of $F$. 
We choose a function $H:S^1\times W\to\R$ and set
\beq
h(x):=\A_H(x)=F(x)+\int_0^1H(t,x(t))dt:\Lambda W\to\R\;.
\eeq
This is the usual action functional of classical mechanics. We have
\beq
\|h-F\| \leq \int_0^1\big[\max_W H(t,\cdot)-\min_W H(t,\cdot)\big]dt=:\|H\|_{\H}\;,
\eeq
i.e., in this case, $\|h-F\|$ is at most the Hofer norm $\|H\|_{\H}$ of the Hamiltonian $H$. Assumption \ref{Ass} is satisfied in this setting, see Section \ref{sec:Floer}, and thus Theorem \ref{thm:main_intro} is applicable.
The condition $\|h-F\|<\S=\infty  $ is in the current situation of course empty and we find that $\A_H=h$ has at least $\mathrm{cuplength}(W)+1$ critical points with critical values in the interval $[-\|H\|_{\H},\|H\|_{\H}]$, see Theorem \ref{thm:fixed points}. Of course, critical points of $\A_H$ are 1-periodic solution of the Hamiltonian equation
\beq
\dot x(t)=X_H(t,x(t)).
\eeq
This is a special case of a Theorem by Floer \cite{Flo89} and independently by Hofer \cite{Hof88}. In fact, they treated the following more general situation using Lagrangian Floer homology.

Let $L\subset W$ be a closed Lagrangian submanifold such that $\om|_{\pi_2(W,L)}=0$.  We consider the space $P_0(W,L):=\{x\in C^\infty([0,1],W)\mid x(0),\,x(1)\in L,\; [x]=0\in\pi_1(W,L)\}$ of paths in $W$ starting and ending on $L$ and which are contractible relative to $L$. Then the symplectic area functional $\A:P_0(W,L)\to\R$ is defined by
\beq
\A(x):=-\int_{D^+}u^*\om
\eeq
where the capping $u:D^+:=\{z\in\C\mid |z|\leq1,\;Im(z)\geq0\}\to W$ is a smooth map with $u|_{S^1\cap D^+}=x$ and $u|_{\R\cap D^+}\subset L$. Again we set $F:=\A$ and $M:=P_0(W,L)$ and observe that $\Crit F$ is the set of constant paths which is identified with $L$, i.e.,
\beq
Z:=\Crit F=L\;.
\eeq
It follows from the fact that $L$ is Lagrangian that $F$ is indeed Morse-Bott along $Z$ and since $F$ has no other critical points we have $\S=\infty$. As above, Hamiltonian perturbations of $F$ give rise to functions $h$ satisfying Assumption \ref{Ass}. That is, we choose a function $H:S^1\times W\to\R$ and set
\beq
h(x):=F(x)+\int_0^1H(t,x(t))dt:P_0(W,L)\to\R\;.
\eeq
Again the condition that $\|h-F\|<\S$ is empty and it follows from Theorem \ref{thm:main_intro} that $\A_H:=h$ has at least $\mathrm{cuplength}(L)+1$ critical points with critical values in the interval $[-\|H\|_{\H},\|H\|_{\H}]$, see Theorem \ref{thm:chords} for the exact statement. Of course, critical points of $\A_H$ are solutions of the Hamiltonian equation with Lagrangian boundary condition given by $L$, i.e.,
\beq
\left\{
\begin{aligned}
&\dot x(t)=X_H(t,x(t))\\
&x(0),\,x(1)\in L
\end{aligned}
\right.
\eeq
of $H$. This is the result previously proved independently by Floer and Hofer. Of course, the result for Lagrangian submanifolds contains the above as special case.

\subsection{Rational symplectic manifolds}

More generally, we may assume that $(M,\omega)$ is symplectically rational, i.e., $\omega|_{\pi_2(W)}=\lambda\Z$ for some $\lambda>0$. If the manifold is symplectically aspherical, we set $\lambda=\infty$. Now the symplectic area functional is defined on the cover $M:=\widetilde{\Lambda W}$ consisting of equivalence classes $\bar x:=[x,u]$ where $x\in\Lambda W$ and $u$ is a capping of $x$. Pairs $(x,u)$ and $(x,v)$ are declared equivalent if $u$ and $v$ have the same symplectic area. Then
\beq
F(\bar x):=\A(\bar x):=-\int_{D^2}u^*\om
\eeq
is well-defined. Critical points of $F$ are of the form $[x,u]$ where $x$ is a constant loop and $u$ is some capping which is topologically a sphere in $W$. Thus the only critical points with critical value zero are the constants with trivial capping meaning the map $u$ is in the equivalence class of the constant capping. This set $Z \subset\Crit F$ is again a Morse-Bott component and $\S=\lambda>0$ since $(W,\om)$ is rational. 

Theorem \ref{thm:main_intro} implies that if $H:S^1\times W\to\R$ is a Hamiltonian function with
\beq
||H||_{\H}<\lambda
\eeq
then the number of 1-periodic Hamiltonian orbits of $H$ with action value in $[-\|H\|_H,\|H\|_H]$ is at least $\mathrm{cuplength}(W)+1$. 

This has been proved by Schwarz in \cite{Sch98}. The analogous statement in the Lagrangian case, with the condition $\omega|_{\pi_2}=\lambda\Z$ replaced by $\om|_{\pi_2(W,L)}=\lambda \Z$ for some positive $\lambda$, is due to Liu, cf. \cite{Liu05}.
In Section \ref{sec:Floer}, we phrase and prove Theorems \ref{thm:fixed points} and \ref{thm:chords} in this more general setting. The aspherical case is then included by setting $\lambda=\infty$.

\begin{Rmk}
The same proof can also be applied in more general settings. We only need that some version of Floer homology is defined and the necessary Assumption \ref{Ass} for compactness is satisfied.
For example, this is the case for symplectic manifolds which are convex at infinity or of bounded geometry.
The same method also applies in the setting of non-resonant magnetic flows on tori, where Frauenfelder, Merry and Paternain have developed a Floer homology, see
\cite{FMP13}. Our Morse theoretic proof also applies in this case and gives a lower bound on the number of periodic orbits. More concretely, the proof shows that if the magnetic field on $T^{2N}$is non-resonant in period $\tau$, the number of $\tau$-periodic orbits is at least $2N+1$.
\end{Rmk}

\subsection{Translated points}
\label{sec:intro_translated}
Now we also apply Theorem \ref{thm:main_intro} to translated points in contact geometry. Translated points were introduced by Sandon in \cite{San11} and related to contact rigidity phenomena. In \cite{San13} Sandon conjectures (and proves in certain cases) that the number of translated points of a contactomorphism of a contact manifold $\Sigma$ is at least equal to the minimal number of critical points of a function on $\Sigma$.  We want to explain how Theorem \ref{thm:main_intro} together with the SFT-type compactness result explained in \cite{AFM13} implies under additional assumptions a lower bound on the number of translated points in terms of $\mathrm{cuplength}(\Sigma)$.

Let $(\Sigma,\alpha)$ be a closed contact manifold and $\varphi:\Sigma\to\Sigma$ be a contactomorphism which is contact isotopic to the identity. Note that $\varphi$ does not preserve the contact form $\alpha$, but we have $\varphi^*\alpha=\rho\alpha$ for some function $\rho:\Sigma\to\R_{>0}$. We call a point $q\in\Sigma$ a translated point with time-shift $\eta\in\R$ if
$$
\left
\{\begin{aligned}
\varphi(q)&=\theta^\eta(q)\\
\rho(q)&=1\;,
\end{aligned}
\right.
$$
where $\theta^t$ is the Reeb flow. We point out that the time-shift is not unique if $q$ lies on a closed Reeb orbit. 
We define the symplectization of $\Sigma$ to be $S\Sigma:=\Sigma\times\R_{>0}$ equipped with the symplectic form $\om:=d(r\alpha)$ where $r\in\R_{>0}$. The unperturbed Rabinowitz action functional $\A:\Lambda S\Sigma\times\R\to\R$ is defined as
\beq
\A(x,\eta):=\int_{S^1}ry^*\alpha-\eta\int_0^1 (r-1)dt
\eeq
where $x=(y,r):S^1\to S\Sigma$. Then $(x,\eta)\in\Crit\A$ if and only if $x(S^1)\subset\Sigma\times\{1\}$ and $\dot y=\eta R(y)$. Here $R$ is the Reeb vector field of $\alpha|_\Sigma$, i.e., we can identify a critical point $(x,\eta)$ with an $\eta$-periodic Reeb orbit where negative $\eta$ means that the $-\eta$-periodic Reeb orbit is traversed in the opposite direction and $\eta=0$ corresponds to a constant loop in $\Sigma$. Constant loops in $\Sigma$ are as usual identified with $\Sigma$ itself. Again the set $Z:=\Sigma$ is a Morse-Bott component, see \cite[Lemma 2.12]{AF10}, and the corresponding spectral gap of $F:=\A$ is $\S=$ minimal period of a contractible Reeb orbit. 

Using the function $\rho:\Sigma\to\R_{>0}$ given by $\varphi^*\alpha=\rho\alpha$, we obtain a non-compactly supported Hamiltonian diffeomorphisms $\phi:S\Sigma\to S\Sigma$ by setting $\phi(q,r):=(\varphi(q),\frac{r}{\rho(q)})$.  As explained in \cite{AFM13}, the Rabinowitz action functional $F=\A$ can be perturbed by a certain cut-off of the lift $\phi$ of $\varphi$ in such a way that the critical points of the perturbed functional $h:=\A_\varphi$ correspond to translated points of $\varphi$, see Lemma 3.5 in \cite{AFM13}. This perturbation is supported inside $\Sigma\times[e^{-\kappa(\varphi)},e^{\kappa(\varphi)}]\subset S\Sigma$, where the constant $\kappa(\varphi)$ is specified in Section \ref{sec:translated points}. This implies that
\beq
\|h-F\|\leq e^{\kappa(\varphi)}\|H\|_{\H}
\eeq
where $H:S^1\times\Sigma\to\R$ is any contact Hamiltonian such that the induced contact isotopy $\psi_t$ satisfies $\varphi=\psi_1$. Thus our main theorem implies that if $e^{\kappa(\varphi)}\|H\|_{\H}<\S$ the functional $\A_\varphi$ has at least $\mathrm{cuplength}(\Sigma)+1$ many critical points with critical values in the interval $[-e^{\kappa(\varphi)}\|H\|_{\H},e^{\kappa(\varphi)}\|H\|_{\H}]$, see Theorem \ref{thm:translated_pts}.

This result complements and strengthens the main result in \cite{AM10} which is concerned with the more general notion of leafwise intersections. Translated points of $\varphi$ correspond to leafwise intersections of its lift $\phi$. More concretely, as we work with the symplectization instead of a symplectic filling of $\Sigma$, our bounds depends only on the contact manifold, where is \cite{AM10}, the lower bound was given by a relative cuplength of $\Sigma$ in the filling. We refer the reader to the end of section \ref{sec:translated points} for further details.

\subsection{Solutions to perturbed Dirac-type equations}
Now we also apply Theorem \ref{thm:main_intro} in the hyperk\"ahler setting, which can be viewed as a generalization of Hamiltonian mechanics with multi-dimensional time.

Let the "time"-manifold be $X$ either $T^3$ or $S^3$ and consider a volume form $\mu$ on $X$ and a divergence-free global frame $v_1, v_2, v_3$. For the considered homology theory to be defined, we need special choices of the global frame, which will be specified in Section \ref{sec:hyperkahler}, see also \cite{HNS09}, where the Floer homology in this setting is constructed. 

Let $Y$ be a compact, flat hyperk\"ahler manifold with almost complex structures $I, J, K$. We define symplectic forms $\omega_i$ by choosing a Riemannian metric and choosing $\omega_1(\cdot,\cdot)=\left<\cdot,I\cdot\right>$, $\omega_2(\cdot,\cdot)=\left<\cdot,J\cdot\right>$ and $\omega_3(\cdot,\cdot)=\left<\cdot,K\cdot\right>$.

In this setting, we consider the manifold $M$ to be the space of smooth maps from $f\colon X\to Y$.
The role of the function $F$ is taken by the functional 
\bea
F(f):=\A(f)=-\sum_{l=1}^3\int_{[0,1]\times X} F^\ast\omega_l\wedge i_{v_l}\mu,
\eea
where $F$ is a homotopy from $f$ to a constant map, i.e., an analog of the capping of a loop in Hamiltonian dynamics.

Critical points of $F$ are solutions to the Dirac-type equation
\bea
\partialslash f:=IL_{v_1}f+JL_{v_2}f+KL_{v_3}f=0.
\eea
All solutions are constant and we identify the Morse-Bott component $Z=Y$, see \cite[Lemmas 2.5 \& 3.7]{HNS09}. As in the case of classical Hamiltonian dynamics described above, we consider a "time"-dependent Hamiltonian on $Y$. That is, the Hamiltonian is a function $H\colon X\times Y\to\R$ and use this to define the Hamiltonian perturbation of $F$ as the functional
\bea
h(f):=\A_H(f)=-\sum_{l=1}^3\int_{[0,1]\times X} F^\ast\omega_l\wedge i_{v_l}\mu-\int_XH(x,f(x))\mu.
\eea
For this functional, critical points are solutions to the equation
\bea
\label{perturbed Dirac-type}
\partialslash f=\nabla H(f).
\eea
Using Fourier analysis in the form of the Peter-Weyl theorem, it was proved in \cite{GH12} that the number of solutions to equation \eqref{perturbed Dirac-type} is bounded below by $\mathrm{cuplength}(Y)+1$ and this result is reproduced by our Morse theoretic proof, see Theorem \ref{thm:cuplength-Dirac}.

\begin{Rmk}
The action functional $\A$ can be defined in a more general setting, see \cite{GH13}, and it is a generalization of the action functional of classical Hamiltonian dynamics. Also the cuplength estimate for the number of critical points holds in this more general situation. However, for our Morse theoretic proof to apply, we need to work in the restricted setting described above as the hyperk\"ahler Floer homology is only defined in this special case.
\end{Rmk}

\section{Preliminaries}
\label{prelim}

In this section, we describe the setting used in the proofs of existence of critical points in the Morse case in more detail. Furthermore, we establish a number of estimates on the energy and the Morse function along trajectories. These estimates are used to show compactness of the moduli spaces considered in the proofs.

\subsection{Setting and basic assumptions}
\label{sec:Setting}

Let $M$ be a manifold, not necessarily compact, and $F\colon M\to\R$ be a smooth function. We assume that $Z\subset M$ is a closed (i.e.~compact and without boundary), connected submanifold satisfying $Z\subset\Crit F$ and $F$ is Morse-Bott along $Z$, that is, for all $z\in Z$ we have
\beq
\ker \Hess_z F=T_zZ\;.
\eeq
For convenience of notation, we also assume that $F|_Z=0$. Furthermore, we assume that the spectral gap $\S$ of $F$ with respect to $Z$ is positive, i.e.
\beq
\S\equiv\S(F,Z):=\inf\{|F(x)|\colon x\in\Crit F\setminus Z\}>0\;.
\eeq
If $F$ does not have critical points outside of $Z$, we set $\S:=\infty$.

%

\begin{Conv}\label{def:beta}
Now define a family of auxiliary functions as follows, see Figure \ref{pic:beta} below.
We first fix two smooth functions $\beta^\pm_\infty\in C^\infty(\R,[0,1])$ satisfying
\begin{enumerate}
\item $\beta^+_\infty(s)=0$ for $s\leq-1$, $\beta^+_\infty(s)=1$ for $s\geq0$ and $\beta^+_\infty$ is monotone increasing,
\item $\beta^-_\infty(s)=1$ for $s\leq0$, $\beta^-_\infty(s)=0$ for $s\geq1$ and $\beta^-_\infty$ is monotone decreasing.
\end{enumerate}
For $k\in\N$ and $r\geq 0$, we fix a smooth family of functions $\beta_r\in C^\infty(\R,[0,1])$ satisfying
\begin{enumerate}
\setcounter{enumi}{2}
\item $\beta_r(s)=0$ for $s\leq-1$ and $s\geq (k+1)r+1$ 
\item $0\leq\beta'_r(s)\leq2$ on $(-1,0)$ and $0\geq\beta'_r(s)\geq-2$ on $((k+1)r,\,(k+1)r+1)$,
\item for $r\geq1$:  $\beta_r(s)=1$ for $s\in[0,(k+1)r]$,
\item for $r\leq1$: $\beta_r'(s)=0$ for $s\in[0,(k+1)r]$, $\lim_{r\to0}\beta_r=0$ in the strong $C^\infty$-topology,\item $\lim_{r\to\infty}\beta_r(s)=\beta_\infty^+(s)$ and $\lim_{r\to\infty}\beta_r(s+(k+1)r)=\beta_\infty^-(s)$ in the $C^\infty_{loc}$ topology.
\end{enumerate} 
\end{Conv}
Later in the proofs, we will specify the value of $k$, but for now, we prove all Lemmas in the general case, where $k$ is any natural number.
\begin{figure}[ht] \label{pic:beta}
\def\svgwidth{80ex} 
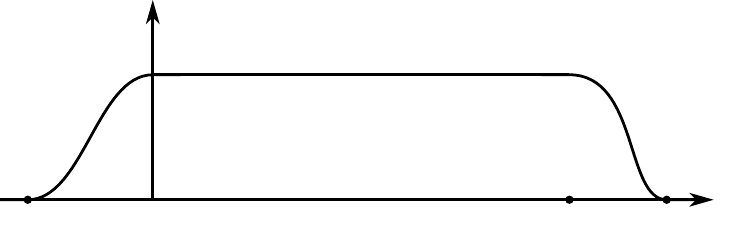
\caption{The function $\beta_r$.}\label{fig:Cn2}
\end{figure}
Let $g$ be a Riemannian metric on $M$ and consider a smooth function $h\colon M\to\R$.
We define the functions 
\beq\label{def_G_rs}
G_{r,s}(x):=\beta_r(s)h(x)+(1-\beta_r(s))F(x) \quad r\geq 0, s\in\R
\eeq 
and the moduli space
\beq
\M:=\Big\{(r,\gamma)\mid r\geq0,\,\gamma\colon\R\to M\text{ satisfies } \gamma'(s)+\nabla^gG_{r,s}(\gamma(s))=0 \text{ and }E(\gamma)<\infty\Big\}.
\eeq
Here, the energy of a gradient flow line is given by 
\beq
E(\gamma):=\int_{-\infty}^\infty|\gamma'(s)|^2ds.
\eeq
Furthermore, we define
\beq
\M_{[0,R]}(Z):=\{(r,\gamma)\in\M\mid r\in[0,R]\text{ and } \lim_{s\to\pm\infty}\gamma(s)\in Z\}\;.
\eeq
In order to guarantee compactness of this moduli space in the $W^{1,2}$-topology, we need to make the following crucial assumption.
This will be satisfied not only in the case of classical Morse-Bott case for suitable Morse-Bott functions as described above, but also in the infinite-dimensional cases of various Floer theories.
Namely, we need to ensure $C^\infty_{loc}$-compactness of the space of trajectories $\gamma$ with finite energy together with all reparametrizations by shifts.

\begin{Assump}
\label{Ass}
For all sequences $\{(r_n,\gamma_n)\}_{n\in\N}\in\M$ and $\{s_n\}_{n\in\N}\subset\R$, the sequence
\beq
\{\gamma_n(\cdot-s_n)\}_{n\in\N}
\eeq
converges (up to taking subsequences) in $C^\infty_{loc}(\R,M)$.
\end{Assump}

\begin{Rmk}
Even though we do not use this in our article, we point out that given this assumption, the Morse cohomology of $F$ is well-defined, see \cite{CF09}. 
Since the manifold $M$ is not necessarily compact, the above compactness assumption is not automatic. 
It is not hard to verify this assumption for Morse-Bott functions $F$ on compact manifolds using the flow equation in the definition of $\M$ and the Arzel\'a-Ascoli theorem.
In the Floer cases, a similar argument using the Floer equation shows the assumption for the analog moduli spaces of Floer trajectories. Detailed arguments for the Floer cases will be given in Section \ref{sec:Ham fixed points}.

In finite dimensions some rather standard assumptions will guarantee Assumption \ref{Ass}. For instance, the assumption holds if the function $F$ is proper and bounded from below, so that sublevel sets are actually compact. This is a condition on $F$ at infinity. For Floer-theoretic generalizations Assumption \ref{Ass} needs additional arguments to exclude bubbling depending on the specific situation.
\end{Rmk}

\subsection{Estimates for Morse-Bott functions}
In this section, we use Assumption \ref{Ass} and the Morse-Bott property to prove a number of estimates along elements from $\M$ and $\M_{[0,R]}(Z)$. Several of the energy computations will also be used later for other estimates needed in the proof of the main theorem. The first is an estimate for the values of $G_{r,s}$ along the flow lines in terms of the oscillation of $h-F$. We recall that the oscillation of a function $f:M\to\R$ is
\beq
||f||:=\sup f-\inf f
\eeq
which is only a seminorm since $||f||=0$ is equivalent to $f$ being constant.
%
%
\begin{Lemma}
\label{bound}
Let $(r,\gamma)\in\M_{[0,R]}(Z)$. Then for all $s\in\R$, we have
\bea
\big|G_{r,s}(\gamma(s))\big|\leq \|h-F\|.
\eea
\end{Lemma}

\begin{proof}
As a first step, we compute for $s\in\R$ the energy of part of the gradient flow line. 
\bea
\label{bound G up}
0\leq E_s(\gamma)&:=\int_s^\infty|\gamma'(t)|^2dt\\
&=-\int_s^\infty\langle\nabla^gG_{r,t}(\gamma(t)),\gamma'(t)\rangle dt\\
&=-\int_s^\infty dG_{r,t}(\gamma(t))[\gamma'(t)]dt\\
&=-\int_s^\infty\frac{d}{dt}G_{r,t}(\gamma(t))dt+\int_s^\infty\frac{\p G_{r,t}}{\p t}(\gamma(t))dt\\
&=-\underbrace{G_{r,+\infty}(\gamma(+\infty))}_{=F(\gamma(+\infty))=0}+G_{r,s}(\gamma(s))+\int_s^\infty\beta_r'(t)\big[h-F\big](\gamma(t))dt\\
&=G_{r,s}(\gamma(s))+\int_s^\infty\beta_r'(t)\big[h-F\big](\gamma(t))dt\;.
\eea 
Similarly, we compute the energy for the front end of the trajectory:

\bea
\label{bound G low}
0\leq E^s(\gamma)&:=\int_{-\infty}^s|\gamma'(t)|^2dt\\
&=\int_{-\infty}^s\langle \gamma'(t),\gamma'(t)\rangle dt\\
&=-\int_{-\infty}^s\langle\nabla^gG_{r,t}(\gamma(s)),\gamma'(t)\rangle dt\\
&=-\int_{-\infty}^s dG_{r,t}(\gamma(t))[\gamma'(s)]dt\\
&=-\int_{-\infty}^s\frac{d}{dt}G_{r,t}(\gamma(t))dt+\int_{-\infty}^s\frac{\p G_{r,t}}{\p t}(\gamma(t))dt\\
&=-G_{r,s}(\gamma(s))+\underbrace{G_{r,-\infty}(\gamma(-\infty))}_{=F(\gamma(-\infty))=0}+\int_{-\infty}^s\beta_r'(t)\big[h-F\big](\gamma(t))dt\\
&=-G_{r,s}(\gamma(s))+\int_{-\infty}^s\beta_r'(t)\big[h-F\big](\gamma(t))dt.
\eea 
Now we can estimate the integral terms to conclude the desired bounds. Since $\beta_r'(t)\neq0$ only for $t\in[-1,0]\cup[(k+1)r,(k+1)r+1]$ with positive resp.~negative sign we  find
\bea
\label{comp:estimate}
G_{r,s}(\gamma(s))&\geq-\int_s^\infty\beta_r'(t)\big[h-F\big](\gamma(t))dt\\ 
&=-\int_{-1}^0\mathbbm{1}_{[t,\infty)}\beta_r'(t)\big[h-F\big](\gamma(t))dt-\int_{(k+1)r}^{(k+1)r+1}\mathbbm{1}_{[t,\infty)}\beta_r'(t)\big[h-F\big](\gamma(t))dt\\
&\geq-\int_{-1}^0\mathbbm{1}_{[t,\infty)}\beta_r'(t)\sup_{x\in M}(h(x)-F(x))dt-\int_{(k+1)r}^{(k+1)r+1}\mathbbm{1}_{[t,\infty)}\beta_r'(t)\inf_{x\in M}(h(x)-F(x))dt\\
&\geq-\int_{-1}^0\beta_r'(t)\sup_{x\in M}(h(x)-F(x))dt-\int_{(k+1)r}^{(k+1)r+1}\beta_r'(t)\inf_{x\in M}(h(x)-F(x))dt\\
&\geq-\sup_{x\in M}(h(x)-F(x))+\inf_{x\in M}(h(x)-F(x))\\
&= -\|h-F\|
\eea
Using now \eqref{bound G low} instead of \eqref{bound G up}, we can use an analogous argument to show
\bea
 G_{r,s}(\gamma(s))\leq \|h-F\|\;.
\eea
This proves the lemma.
\end{proof}

The computations \eqref{bound G up} and \eqref{bound G low} in this proof give rise to the following energy estimates for trajectories in $\M_{[0,R]}(Z)$. Namely, consider \eqref{bound G low} for large $s$. Then we have $G_{r,s}=F$ and $\gamma(s)\stackrel{s\to\pm\infty}{\xrightarrow{\hspace*{5ex}} } Z$ and thus
\bea
E^s(\gamma)&=-F(\gamma(s))+\int_{-\infty}^s\beta_r'(t)\big[h-F\big](\gamma(t))dt\\
&=-F(\gamma(s))+\int_{-1}^0\beta_r'(t)\big[h-F\big](\gamma(t))dt+\int_{(k+1)r}^{(k+1)r+1}\beta_r'(t)\big[h-F\big](\gamma(t))dt\\
\eea
as these are the only intervals with $\beta'\neq0$. Then the signs of $\beta'$ in these intervals show that
\bea
\label{energy}
0\leq E(\gamma)\leq\sup_{x\in M}(h(x)-F(x))-\inf_{x\in M}(h(x)-F(x))=\|h-F\|.
\eea
The first term in the computation of the energy vanishes in the limit as $\gamma(s)\to Z\subset F^{-1}(0)$.
This energy estimate will be used in the Floer cases to exclude bubbling-off in the proof of compactness of the moduli spaces analog to $\M_{[0,R]}(Z)$ in the Morse case.

To prove below that trajectories cannot leave a certain neighborhood of $Z$, we will need an estimate of the function $F$ along the trajectories in $\M_{[0,R]}(Z)$.
Here, we use some of the above computations to show the following.
\begin{Lemma}
\label{bound F}
Along curves $\gamma$ with $(\gamma,r)\in\M_{[0,R]}(Z)$, we have the following estimate:
\beq
\big|F\big(\gamma(s)\big)\big|\leq\|h-F\|
\eeq
for all $s\in\R$.
\end{Lemma}

\begin{proof}
We start with the case that $s\leq (k+1)r$ and use estimate \eqref{bound G low}.
\bea
G_{r,s}(\gamma(s))&\leq\int_{-\infty}^s\underbrace{\beta_r'(t)}_{\geq0}\big[h-F\big](\gamma(t))dt\\ 
&\leq\int_{-\infty}^s\beta_r'(t)\sup(h-F)dt\\ 
&=\beta_r(s)\sup(h-F)
\eea
Now we expand the expression of $G_{r,s}$ (see \eqref{def_G_rs}) to obtain
\bea
\beta_r(s)\sup(h-F)&\geq G_{r,s}(\gamma(s))\\
&=\beta_r(s)h(\gamma(s))+(1-\beta_r(s))F(\gamma(s))\\
&=\beta_r(s)\big[h-F\big](\gamma(s))+F(\gamma(s))\\
\eea
and thus
\bea
F(\gamma(s))&\leq\beta_r(s)\Big\{\sup(h-F)-\big[h-F\big](\gamma(s))\Big\}\\
&\leq\underbrace{\beta_r(s)}_{0\leq\bullet\leq1}\Big\{\sup(h-F)-\inf(h-F)\Big\}\\
&\leq||h-F||\;.
\eea
Using estimate \eqref{bound G up} instead of \eqref{bound G low} we derive
\bea
G_{r,s}(\gamma(s))&\geq-\int_s^\infty\beta_r'(t)\big[h-F\big](\gamma(t))dt\\
&=-\int_s^{(k+1)r}\underbrace{\beta_r'(t)}_{\geq0}\big[h-F\big](\gamma(t))dt-\int_{(k+1)r}^\infty\underbrace{\beta_r'(t)}_{\leq0}\big[h-F\big](\gamma(t))dt\\
&\geq-\int_s^{(k+1)r}\beta_r'(t)\sup(h-F)dt-\int_{(k+1)r}^\infty\beta_r'(t)\inf(h-F)dt\\
&=\big(\beta_r(s)-1\big)\sup(h-F)+\inf(h-F)\\
&=\beta_r(s)\sup(h-F)-||h-F||
\eea
and continue as above
\beq
\beta_r(s)\sup(h-F)-||h-F||\leq G_{r,s}(\gamma(s))=\beta_r(s)\big[h-F\big](\gamma(s))+F(\gamma(s))\\
\eeq
to obtain
\bea
F(\gamma(s))&\geq\beta_r(s)\sup(h-F)-||h-F|| -\beta_r(s)\big[h-F\big](\gamma(s))\\
&=\underbrace{\beta_r(s)}_{\geq0}\Big\{\underbrace{\sup(h-F)-\big[h-F\big](\gamma(s))}_{\geq0}\Big\}-||h-F||\\
&\geq-||h-F||\;.
\eea
This proves the required estimate $|F(\gamma(s))|\leq\|h-F\|$ for $s\leq (k+1)r$. We now will prove the same estimate for $s\geq(k+1)r$. We start with \eqref{bound G up} and get
\bea
G_{r,s}(\gamma(s))&\geq-\int_s^\infty\underbrace{\beta_r'(t)}_{\leq0}\big[h-F\big](\gamma(t))dt\\
&\geq\beta_r(s)\inf(h-F)
\eea
and therefore
\beq
\beta_r(s)\inf(h-F)\leq G_{r,s}(\gamma(s))=\beta_r(s)\big[h-F\big](\gamma(s))+F(\gamma(s))\;.
\eeq
leading to
\bea
F(\gamma(s))&\geq\beta_r(s)\Big\{\inf(h-F)-\big[h-F\big](\gamma(s))\Big\}\\
&\geq-\beta_r(s)||h-F||\\
&\geq-||h-F||\;.
\eea
Using \eqref{bound G low} we obtain
\bea
G_{r,s}(\gamma(s))&\leq\int_{-\infty}^s\beta_r'(t)\big[h-F\big](\gamma(t))dt\\ 
&=\int_{-\infty}^{(k+1)r}\underbrace{\beta_r'(t)}_{\geq0}\big[h-F\big](\gamma(t))dt+\int_{(k+1)r}^s\underbrace{\beta_r'(t)}_{\leq0}\big[h-F\big](\gamma(t))dt\\ 
&\leq \sup(h-F)+\big(\beta_r(s)-1\big)\inf(h-F)
\eea
thus
\beq
\sup(h-F)+\big(\beta_r(s)-1\big)\inf(h-F)\geq G_{r,s}(\gamma(s))=\beta_r(s)\big[h-F\big](\gamma(s))+F(\gamma(s))
\eeq
and
\bea
F(\gamma(s))&\leq||h-F||+\underbrace{\beta_r(s)}_{\geq0}\Big\{\underbrace{\inf(h-F)-\big[h-F\big](\gamma(s))}_{\leq0}\Big\}\\
&\leq||h-F||\;.
\eea
This finishes the proof.
\end{proof}

For the following we define the energy density of $F$ by
\beq
e(x):=\left<\nabla F(x),\, \nabla F(x)\right>\text{ for } x\in M\;.
\eeq
\begin{Lemma}[Action-energy estimate]
\label{Morse-Bott-action-energy}
If the function $F$ is Morse-Bott along the closed critical submanifold $Z$ it satisfies an action-energy estimate, i.e.~there exists a tubular neighborhood $U_Z$ of $Z$ and a constant $C>0$ such that
\beq
|F(x)|\leq C\ e(x)\quad\forall x\in U_Z.
\eeq
\end{Lemma}

Recall that $F|_Z=0$ and that this holds in the Floer cases due to the isoperimetric inequality and H\"older.

\begin{proof}
We identify a neighborhood $U$ of $Z$ with a neighborhood of the zero-section in the normal bundle of $Z$ in $M$. We write $x=(z,y)$ for $x\in U$ with $z\in Z$ and $y$ being the normal component.
Now we estimate $e(x)$ and $F(x)$ separately.
For $e(x)$, we use the Morse-Bott condition to find for $y$ sufficiently small
\bea
e(x)&=\left<\nabla F(x),\, \nabla F(x)\right>\\
&=|\Hess\, F(z)(y)|^2+O(|y|^3)\\
&\geq C_1\,|y|^2+O(|y|^3)\\
&\geq C_2\left(1-|y|\right)|y|^2\;.
\eea
Similarly, we find for $F(x)$, using again the Morse-Bott condition, $F(Z)=0$, and $y$ sufficiently small
\bea
F(x)&=\tfrac{1}{2}D^2F(z)(y,y)+O(|y|^3)\\
\Rightarrow |F(x)|&\leq C_3\,|y|^2+O(|y|^3)\\
&\leq C_4\left(1+|y|\right)|y|^2.
\eea
We conclude that
\bea
|F(x)|\leq \frac{C_4\left(1+|y|\right)}{C_2\left(1-|y|\right)}\ e(x)\leq C\ e(x)
\eea
holds on a neighborhood $U_Z$.
\end{proof}

\begin{Lemma}
\label{exp convergence}
We fix $\gamma\in\M$ and assume that there is a sequence $s_n\to\infty$ such that 
\beq
\lim_{n\to\infty}\gamma(s_n)=z^+\in Z.
\eeq
 Moreover, we require that there exists $\bar{s}\in\R$ such that 
\beq
\gamma(s)\in U_Z\quad\forall s\geq\bar{s}\;.
\eeq
Then the gradient flow line $\gamma$ converges exponentially fast to $z^+$. I.e.~for $s\geq\bar{s}$ we find constants $A\geq0$ and $B>0$ such that the estimate
\bea
d(z^+,\gamma(s))\leq \sqrt{F(\gamma(\bar{s}))}\,Ae^{-Bs}
\eea
holds for all $s\geq\bar{s}$. We point out that the constants $A,B$ are independent of $\gamma$. An analogous statement for $s_n\to-\infty$ holds.
\end{Lemma}

\begin{proof}
The first step to prove this lemma is an estimate of the distance to $Z$ in terms of the value of $F$ along the flow line $\gamma$. Since $\gamma(s_n)$ converges to $z^+\in Z$ we find for any fixed $s\geq\bar{s}$ an integer $N\in\N$ with $s_N>s$ and 
\bea
d(z^+,\gamma(s_n))<\frac{1}{N}
\eea
for all $n> N$. The following estimate is originally taken from \cite{AF13} and is used here with only minor changes.
We use some $n> N$ and the action-energy estimate from Lemma \ref{Morse-Bott-action-energy} to compute (recalling that $s_n>s$)
\bea
 d(z^+,\gamma(s))&\leq d(z^+,\gamma(s_n)) + d(\gamma(s_n),\gamma(s))\\
 &<\frac{1}{N}+\int_s^{s_n}|\gamma'(t)|\, dt\\
 &=\frac{1}{N}+\int_s^{s_n}|\nabla F(\gamma(t))|\, dt\\
 &=\frac{1}{N}+\int_s^{s_n}\frac{|\nabla F(\gamma(t))|^2}{|\nabla F(\gamma(t))|}\, dt\\
&=\frac{1}{N}+\int_s^{s_n}\frac{|\nabla F(\gamma(t))|^2}{\sqrt{e(\gamma(t))}}\, dt\\
&\leq\frac{1}{N}+\sqrt C \int_s^{s_n}\frac{|\nabla F(\gamma(t))|^2}{\sqrt{F(\gamma(t))}}\, dt\\
&=\frac{1}{N}-\sqrt C \int_s^{s_n}\frac{\left<\nabla F(\gamma(t)),\gamma'(t)\right>}{\sqrt{F(\gamma(t))}}\, dt\\
&=\frac{1}{N}-\sqrt C \int_s^{s_n}\frac{\frac{d}{dt}F(\gamma(t))}{\sqrt{F(\gamma(t))}}\, dt\\
&=\frac{1}{N}-2\sqrt C \int_s^{s_n}\frac{d}{dt}\sqrt{F(\gamma(t))}\, dt\\
&=\frac{1}{N}-2\sqrt C \left(\sqrt{F(\gamma(s_n))}-\sqrt{F(\gamma(s))}\right)\\
\eea
As the left hand side is independent of $N$, we can now take the limit $N\to\infty$. Then $F(\gamma(s_n))$ converges to 0, since we assumed $n>N$ and $\gamma(s_n)$ converges to $z^+\in Z$.
In the limit, we find the estimate
\bea
 d(z^+,\gamma(s))\leq 2\sqrt C \sqrt{F(\gamma(s))}\quad\forall s\geq\bar{s}.
\eea
Now it suffices to calculate
\bea
\frac{d}{ds}F(\gamma(s))&=-\left<\nabla F(\gamma(s)),\, \nabla F(\gamma(s))\right>\\
&=-e(\gamma(s))\\
&\leq -\frac{1}{\sqrt C } F(\gamma(s)).
\eea
This differential inequality shows that
\bea
F(\gamma(s))\leq F(\gamma(\bar{s}))\, e^{-\frac{s}{\sqrt C}}
\eea
and therefore the above computation gives the estimate
\bea
 d(z^+,\gamma(s))&\leq 2\sqrt{CF(\gamma(s))}\leq 2\sqrt{CF(\gamma(\bar{s}))}e^{-\frac{s}{2\sqrt C}}.
\eea
This shows exponential convergence of $\gamma(s)$ to $z^+\in Z$ and completes the proof of the lemma.
\end{proof}

\subsection{Compactness of moduli spaces}
\label{compactness}

In this subsection, we will apply Assumption \ref{Ass}, the Morse-Bott property and the bounds given by Lemma \ref{bound} to prove $W^{1,2}$-compactness of $\M_{[0,R]}(Z)$.

\begin{Thm}
\label{thm:compactness}
If 
\beq
\|h-F\|<\S
\eeq
then for all $R\geq0$ the moduli space $\M_{[0,R]}(Z)$ is compact in $W^{1,2}(\R,M)$.
\end{Thm}

\begin{proof}[Proof of Theorem \ref{thm:compactness}]
Let $(r_n,\gamma_n)$ be a sequence in $\M_{[0,R]}(Z)$. 
As a first step, we show that (a subsequence of) the sequence $(\gamma_n)$ converges in $C^\infty_{loc}(\R,M)$ and its limit is again in $\M_{[0,R]}(Z)$. This then implies convergence of $(\gamma_n)$ converges in $C^\infty(\R,M)$.

Setting $s_n=0$ in Assumption \ref{Ass}, we see that a subsequence of $\gamma_n$ converges in $C^\infty_{loc}(\R,M)$ to some limit $\gamma$, since $\M_{[0,R]}(Z)\subset\M$.
The compactness of $\left[0,R\right]$ implies that, by possibly passing to another subsequence, $(r_n,\gamma_n)$ converges to $(r,\gamma)$ with $r\in\left[0,R\right]$.

We will show that $(r,\gamma)\in\M_{[0,R]}(Z)$, i.e., 
\begin{itemize}
\item $(r,\gamma)\in\M$ and
\item $\lim_{s\to\pm\infty}\gamma(s)\in Z$.
\end{itemize}
For the first assertion, we need to show that $\gamma'=\nabla G_{r,s}(\gamma)$ and $E(\gamma)<\infty$.
The gradient flow equation is a local condition and therefore follows from $C^\infty_{loc}$-convergence of the sequence $\gamma_n$.
To show that the energy of the limit is finite, observe that the energy values $E(\gamma_n)$ are uniformly bounded by $\|h-F\|$. Then also the energy of the limit cannot exceed this bound and is therefore finite. 

The second assertion requires more work. We consider only the limit as $s\to+\infty$ as the case $s\to-\infty$ is analogous. The first step is to find a candidate for the limit of $\gamma(s)$ as $s\to\infty$ and then show that this point is indeed the limit and lies in $Z$. For this, we choose an arbitrary sequence $(s_n)\to\infty$ such that $s_n>(k+1)R+1$ for all $n\in\N$.
We define the new sequence $\bar\gamma_n$ by 
$$
\bar\gamma_n(\cdot)=\gamma(\cdot+s_n).
$$
For this sequence (and the constant sequence $r_n=r$), we apply again Assumption \ref{Ass} to find a subsequence of $\bar\gamma_n$ which converges to some gradient flow line $\gamma^+$. 
Now we can apply an argument as in Lemma 2.1 in \cite{CF09} to show that $\gamma^+$ is constant. For the sake of completeness, we include the argument here.

Assume $\gamma^+$ to be non-constant and observe that if $s$ is positive, $\gamma^+$ is a gradient flow line of $F$ by our choice of $s_n$. Since $\gamma^+$ is a gradient flow line, there is some $s_0>0$, such that
$$
F(\gamma^+(0))-F(\gamma^+(s_0))=\epsilon
$$
for some $\epsilon >0$.
By $C^\infty_{loc}$-convergence of $\bar\gamma_n$ to $\gamma^+$, this implies for some $n_0\in\N$ that 
$$
F(\bar\gamma_n(0))-F(\bar\gamma_n(s_0))\geq\epsilon/2
$$
for all $n\geq n_0$.
We now use the definition of $\bar\gamma_n$ to find
$$
F(\gamma(s_n))-F(\gamma(s_0+s_n))\geq\epsilon/2.
$$
We now define a subsequence of $s_n$ starting at the above $n_0$, as
$$
n_k=\min\{n\, |\, n_k-n_{k-1}\geq s_0\}.
$$
If we choose $k_0>2E(\gamma)/\epsilon$, we can now compute for the energy $E(\gamma)$ of the gradient flow line $\gamma$:
\bea
E(\gamma)&=\int_{-\infty}^\infty |\partial_s\gamma|^2\ ds\\
&\geq\sum_{k=0}^{k_0-1}\int_{s_{n_k}}^{s_{n_k}+s_0} |\partial_s\gamma|^2\ ds\\
&=-\sum_{k=0}^{k_0-1}\int_{s_k}^{s_{n_k}+s_0} \frac{d}{d s}F(\gamma(s))\ ds\\
&\geq\sum_{k=0}^{k_0-1}F(\gamma(s_{n_k}))-F(\gamma(s_{n_k}+s_0))\\
&\geq\frac{k_0\epsilon}{2}\\
&>E(\gamma).
\eea
We conclude that $\gamma^+$ is indeed constant and thus a critical point  of $F$.
Now the assumption $\|h-F\|<\S$ together with Lemma \ref{bound} implies that $\gamma^+=:z^+=\lim_{n\to\infty}\gamma(s_n)\in Z$.

We now show that $\gamma$ converges not only along the sequence $(s_n)$ but actually $\displaystyle\lim_{n\to\infty}\gamma(s)=z^+$. This follows immediately from Lemma \ref{exp convergence} once we find $\bar s\in\R$ such that $\gamma_n([\bar{s},\infty))\subset U_Z$, where $U_Z$ is as in Lemma \ref{Morse-Bott-action-energy}. We fix a neigborhood $V_Z\subset U_Z$ of $Z$ together with constants $d,D>0$ such that $0<d\leq\|\nabla F\| \leq D$ holds on $U_Z\setminus V_Z$, see the proof of Lemma \ref{Morse-Bott-action-energy}. This is possible since $F$ is Morse-Bott along $Z$. If no $\bar{s}$ as claimed exists then there is a sequence $t_n\to\infty$ such that $\gamma(t_n)\notin U_Z$. But as $\gamma(s_n)\to z^+\in Z$, we may assume that $\gamma(s_n)\in V_Z$. We may assume that $s_n<t_n<s_{n+1}$ by passing to subsequences if needed. 

The inequality $0<d\leq\|\nabla F\| \leq D$ on $U_Z\setminus V_Z$ implies that for every interval $(s_n,t_n)$ and $(t_n,s_{n+1})$ the energy of $\gamma$ increases by at least some fixed positive amount. This contradicts the fact that $\gamma$ has finite energy. Thus there exists some $\bar s$ such that $\gamma(s)\in U_Z$ for all $s\geq \bar s$ and by Lemma \ref{exp convergence}, $\gamma$ converges exponentially fast to $z^+\in Z$ and hence $\gamma\in\M_{[0,R]}(Z)$.

Now it remains to show that $\gamma_n\stackrel{W^{1,2}}{\xrightarrow{\hspace*{3ex}}}\gamma$. Following \cite[Lemma 2.39]{Sch93}, it suffices to show that $\gamma_n\to\gamma$ uniformly in $C^0$-topology. The proof in \cite{Sch93} relies on the fact that the set
\beq
K:=\overline{\bigcup_{n\in\N}\gamma_n(\R)}\subset M
\eeq
is compact. This is inferred from a Palais-Smale type assumption which we do not make here. Instead we can use Assumption \ref{Ass} to arrive at the same conclusion. Indeed, any sequence $\big(\gamma_k(t_k)\big)$ has a convergent subsequence since $\big(\gamma_k(\cdot+t_k)\big)$ converges in $C^\infty_{loc}$ according to Assumption \ref{Ass}. Then the proof in \cite{Sch93} applies verbatim. We recollect the salient steps.

The compactness of $K$ implies that $\nabla G_{r,s}|_K$ is uniformly Lipschitz continuous independent of $r$ and $s$. Thus there exists a constant $\tilde c$ such that
\bea
\big|\, \|\nabla G_{r_n,s}(\gamma_n(t))\| - \|\nabla G_{r_n,s}(\gamma_n(s))\|\, \big|&\leq \tilde c\cdot d(\gamma_n(t),\gamma_n(s))\\
&\leq \tilde c\cdot \left|\int_s^t|\dot\gamma_n(\tau)|d\tau\right|\\
&\leq \tilde c\sqrt{|t-s|} \sqrt{\int_s^t|\dot\gamma_n(\tau)|^2d\tau}\\
&\leq \tilde c\sqrt{|t-s|} \sqrt{\int_s^t|\nabla G_{r_n,\tau}(\gamma_n(\tau))|^2d\tau}\\
\eea
We focus on the last factor. 
\bean
\sqrt{\int_s^t|\nabla G_{r_n,\tau}(\gamma_n(\tau))|^2d\tau}&= \sqrt{\int_s^t\frac{d}{d\tau} G_{r_n,\tau}(\gamma_n(\tau)) - \frac{\p G_{r_n,\tau}}{\p \tau}(\gamma_n(\tau))d\tau}\\
&=\sqrt{\big[G_{r_n,t}(\gamma_n(t))-G_{r_n,s}(\gamma_n(s))\big]-\int_s^t\frac{\p G_{r_n,\tau}}{\p \tau}(\gamma_n(\tau))d\tau}\\
&= \sqrt{\big[G_{r_n,t}(\gamma_n(t))-G_{r_n,s}(\gamma_n(s))\big]-\int_s^t\beta_{r_n}'(t)\big[h-F\big](\gamma_n(\tau))d\tau}\\
\eea
Since $K$ is compact, $\beta_r'$ is bounded (see Convention \eqref{def:beta}) and non-zero only on a compact set, and $(r_n)$ is bounded the above term is bounded as well. Therefore there is another constant $c$ such that
\beq
\big|\, \|\nabla G_{r_n,s}(\gamma_n(t))\| - \|\nabla G_{r_n,s}(\gamma_n(s))\|\, \big|\leq c\sqrt{|t-s|} \quad\forall s,t\in\R,\forall n\in\N
\eeq
Assume now that $(\gamma_n)$ does not uniformly $C^0$-converge to $Z$ for $s\to\infty$. Then there exists a sequence $s_k\to\infty$ such that $\gamma_{n_k}(s_k)\not\in U_Z$ (after possibly shrinking $U_Z$.) Since $s_k\to\infty$ we have $G_{r_{n_k},s_k}(\gamma_{n_k}(s_k))=F(\gamma_{n_k}(s_k))$ for $k$ sufficiently large. In particular, we may assume that for all $k\in\N$ we have $|\nabla F(\gamma_{n_k}(s_k))|\geq\epsilon$ for some fixed $\epsilon>0$.  As in the proof of \cite[Lemma 2.39]{Sch93}, this implies that there exists $\delta>0$ such that $F(\gamma_{n_k}(s_k))\geq\delta>0=F(Z)$. On the other hand, the exponential convergence of $\gamma$ to $Z$ implies that there is some $\tilde s$ such that $F(\gamma(s))\leq\frac{1}{2}\delta$ for all $s\geq\tilde s$. Then the $C^\infty_{loc}$-convergence of $\gamma_n$ to $\gamma$ yields a contradiction to $s_k\to\infty$, as the gradient flow lines can only decrease the values of $F$. Thus, there exists $\bar s$ such that 
\beq
\gamma_n([\bar s,+\infty))\subset U_Z \quad\forall n\in\N\;.
\eeq
The same argument works, of course, for $s\to-\infty$.

Now we can show the uniform $C^0$-convergence of $\gamma_n$ to $\gamma$. From Lemma \ref{exp convergence} and the $C^\infty_{loc}$-convergence of $\gamma_n$ to $\gamma$ we conclude that $\gamma_n$ converge uniformly exponentially fast to $Z$ at $s\to \pm\infty$ on $(-\infty, -\bar s]\cup[\bar s,+\infty)$. In particular, $(\gamma_n)$ converges uniformly in $C^0$ on $(-\infty, -\bar s]\cup[\bar s,+\infty)$. On the remaining compact interval $[-\bar s,\bar s]$ the $C^\infty_{loc}$-convergence, of course, implies uniform $C^0$-convergence. Together we obtain uniform $C^0$-convergence on all of $\R$.

This finishes the proof, as the desired $W^{1,2}$-convergence follows from $C^\infty_{loc}$-convergence and uniform $C^0$-convergence as in \cite[Lemma 2.39]{Sch93}.

%

\end{proof}

\begin{Rmk}
In Morse theory it is convenient to work with $W^{1,2}$-spaces. However, in Floer theory $W^{1,p}$-spaces, $p>2$, are necessary. We point out that the above argument works unchanged for $W^{1,p}$ due to the uniform exponential convergence, see \cite[Proposition 4.3.11]{Sch95}.
\end{Rmk}

\section{Cup-length estimates}
\label{crit points}

In this section, we will prove Theorem \ref{thm:main_intro}. From now on, we fix the function $F$ having a Morse-Bott critical submanifold $Z$. For convenience, we assume again that $F|_Z=0$. Assume also that $F$ has positive spectral gap $\S>0$. Moreover, let $h$ be a smooth function with
\beq
\|h-F\|<\S\;.
\eeq 
We require that Assumption \ref{Ass} holds. Recall that we need to show that the function $h$ has at least $\mathrm{cuplength}(Z)+1$ critical points with critical values in the interval $[-\|h-F\|,\,\|h-F\|]$. We further recall that the cuplength of $Z$ is defined as
\beq
\mathrm{cuplength}(Z):=\max\{k\in\N\mid \exists a_1,\ldots,a_k\in\H^{\geq1}(Z)\text{ such that } a_1\cup\ldots\cup a_k\neq0\}.
\eeq
Here, $\H^{\geq1}(Z)$ denotes the cohomology in degree at least 1. In the following we will use Morse (co-)homology for closed finite dimensional manifolds. A detailed treatment can be found in \cite{Sch93}. We use here the symbols $\CM_*(f)$ resp. $\CM^*(f)$ for the Morse chain resp.~cochain complex of the Morse function $f$.

\begin{proof}
We follow the line of proof from \cite{AM10}. We set $k:=\mathrm{cuplength}(Z)$ and choose Morse functions $f_1,\ldots,f_k,f_\ast:Z\to\R$. We extend $f_i$ to Morse functions $\bar{f}_i:M\to\R$ such that on a tubular neighborhood $U$ of $Z$ we have ${\bar{f}_i}|_{U}=f_i+q$ where $q:U\to\R$ is a positive definite quadratic form in normal direction to $Z$, thus, $\Crit f_i\subset\Crit\bar{f}_i$. Moreover we assume that $\bar{f}_i$ is proper and bounded from below. This implies that the negative gradient flow of $\bar{f}_i$ is defined in forward time. Finally, we choose Riemannian metrics $g_1,\ldots, g_k,g_\ast$ on $M$.

For a fixed $R\geq0$, the set 
\beq
\M_R(Z)=\big\{\gamma \,|\,(R,\gamma)\in\M_{[0,R]}(Z)\big\}
\eeq
is the zero-set of a Fredholm section in a separable Banach space bundle. Lemma \ref{thm:compactness} of compactness of $\M_{[0,R]}(Z)$ applies to $\M_R(Z)$ with the same proof. 

We recall the following general fact. Let $\F$ be a Fredholm section of a separable Banach space bundle having a compact zero-set $\{\F=0\}$. Let $s$ be a section with compact linearization, in particular, $\F+s$ is Fredholm again. If $s$ is sufficiently small then the set $\{\F+s=0\}$ is compact, too. Moreover, for a generic such section $s$ the Fredholm operator $\F+s$ is transverse to the zero-section, see for instance \cite[Theorems 5.5 \& 5.13]{HWZ14}. We call $s$ an abstract perturbation. 

Thus, we can choose a small abstract perturbation of the Fredholm section so that the zero-set $\widetilde\M_R(Z)$ of the perturbation is a compact, smooth manifold of finite dimension. Moreover, by the Morse-Bott assumption on $F$, the Fredholm section is already transverse for $R=0$ and thus, by compactness, also for small $R\geq0$. Therefore, we may assume that $\widetilde\M_R(Z)=\M_R(Z)$ for $R$ sufficiently close to $0$. Analogously, we can perturb the moduli space $\M_{[0,R]}(Z)$ to obtain a smooth compact manifold $\widetilde\M_{[0,R]}(Z)$. Again, we may assume that this perturbation is chosen such that the fibers of $\widetilde\M_{[0,R]}(Z)$ over $0$ and $R$ with respect to the natural projection to $[0,R]$ agree with $\M_0(Z)$ and $\widetilde\M_R(Z)$, respectively. 

The space $\widetilde\M_R(Z)$ carries a natural evaluation map $\widetilde\ev_R\colon\widetilde\M_R(Z)\to M^k$ defined by
\beq
\label{eval}
\widetilde\ev_R(\gamma)=\left( \gamma(R), \gamma(2R),\ldots, \gamma(kR)\right).
\eeq
We point out that $\widetilde\ev_{R=0}$ is the diagonal embedding of $Z$ into $M^k$. Indeed, for critical points $x_i\in\Crit f_i\subset\Crit \bar{f}_i$, and $x_\ast^\pm\in\Crit f_\ast$, consider the moduli space
\bean
\widetilde\M(R,x_1,\ldots,x_k,x_\ast^-,x_\ast^+):=\left\{
\gamma\in\widetilde\M_R(Z)\left|
\begin{array}{c}
\gamma(-\infty)\in W^u(x_\ast^-,f_\ast),\gamma(\infty)\in W^s(x_\ast^+,f_\ast),\\
\widetilde{\ev}_R(\gamma)\in W^s(x_1,\bar{f}_1)\times\ldots\times W^s(x_k,\bar{f}_k)
\end{array}\right. \right\}
\eea 
where $W^u$ and $W^s$ are unstable and stable manifolds, respectively. Above evaluation maps are restrictions of natural evaluation map on the Banach manifold containing $\widetilde\M_R(Z)$. On this Banach manifold the evaluation maps are submersions. Thus we may assume that our perturbations are chosen such that all $\widetilde\M(R,x_1,\ldots,x_k,x_\ast^-,x_\ast^+)$ are smooth manifolds.

We recall that for $R=0$ the space $\M_{R=0}(Z)$ is the space of constant maps to $Z$. Thus the spaces $\widetilde\M(R=0,x_1,\ldots,x_k,x_\ast^-,x_\ast^+)$ can be identified with $W^s(x_1,f_1)\cap\ldots\cap W^s(x_k,f_k)\cap W^s(x_\ast^+,f_\ast)\cap W^u(x_\ast^-,f_\ast)$ since $\bar{f}_i$ is a quadratic extension of $f_i$ near $Z$.

We now define cohomology operations on the Morse co-chain groups
\bea
\theta_R\colon\CM^\ast(f_1)\otimes\ldots\otimes\CM^\ast(f_k)\otimes\CM_\ast(f_\ast)&\to\CM_\ast(f_\ast)\\
x_1\otimes\ldots\otimes x_k\otimes x_\ast^-&\mapsto\sum_{x_\ast^+\in\Crit(f_\ast)}\#_2\widetilde\M(R,x_1,\ldots,x_k,x_\ast^-,x_\ast^+)\cdot x_\ast^+.
\eea
Here we use the convention that $\#_2\widetilde\M(R,x_1,\ldots,x_k,x_\ast^-,x_\ast^+)$ is the parity of this set if it is $0$-dimensional and $0$ otherwise.

For $R=0$ the argument in \cite{AM10}, see also \cite{Sch93}, shows that this is a Morse-theoretic realization of the usual cup-product in cohomology.  Namely, under the identifications of $\HM^*(f_i)$ with $\H^*(Z)$ and $\HM_*(f_\ast)$ with $\H_*(Z)$, in cohomology the map $\theta_0$ agrees with 
\bea
\Theta\colon \H^*(Z)\otimes\ldots\otimes \H^*(Z)\otimes \H_*(Z)&\to \H_*(Z)\\
a_1\otimes\ldots\otimes a_k\otimes b&\mapsto (a_1\cup\ldots \cup a_k)\cap b.
\eea
We point out that the map $\theta_R$ is chain homotopy equivalent to $\theta_0$. For this we recall that $\widetilde\M_{[0,R]}(Z)$ is a compact manifold with boundary $\M_0(Z)\cup\widetilde\M_R(Z)$. Moreover, the space $\widetilde\M_{[0,R]}(Z)$ carries an evaluation map extending $\widetilde\ev_0$ and $\widetilde\ev_R$. This gives rise to a chain homotopy operator between $\theta_0$ and $\theta_R$. 

We now use the cohomology operations to prove the theorem. We assume that $h$ has only finitely many critical points since otherwise there is nothing to prove. Now, we require additionally that the Morse functions $f_i:Z\to\R$ (and their extensions $\bar{f}_i$) satisfy
\beq\label{clever_choice_of_f_i}
W^s(x_i,\bar{f}_i)\cap\Crit (h)=\emptyset
\eeq 
for all $x_i\in\Crit f_i\subset\Crit \bar{f}_i$ with non-zero Morse index. We point out that this a generic (even open and dense) condition for $f_i$ since $\Crit(h)$ is finite. Since $k=\mathrm{cuplength}(Z)$ we find cohomology classes $a_1,\ldots,a_k\in\H^{\geq1}(Z)$ with $a_1\cup\ldots\cup a_k\neq0$. In particular, the map $\Theta$ does not vanish. Since, in cohomology $\theta_0$ agrees with $\Theta$ and $\theta_R$ is chain homotopic to $\theta_0$ we conclude that the map $\theta_R$ does not vanish for all $R\geq0$. This implies that for each $n\in\N$ there are critical point $x_i$ (a priori depending on $n$) of $f_i$ and $x_\ast^\pm$ of $f_\ast$ of index $\geq 1$, such that
\beq
\widetilde\M(n,x_1,\ldots,x_k,x_\ast^-,x_\ast^+)\neq\emptyset.
\eeq
Since there are only finitely many critical points, we may assume that $x_i$ are independent of $n$.
If the unperturbed moduli space $\M(n,x_1,\ldots,x_k,x_\ast^-,x_\ast^+)$ were empty, the same were true for small perturbations. Hence, we can choose elements 
\beq
\gamma_n\in\M(n,x_1,\ldots,x_k,x_\ast^-,x_\ast^+)
\eeq
and consider the sequences
\beq
\gamma_{n,j}(\cdot):=\gamma_n(\cdot+nj)
\eeq
for $j=0,\ldots, k+1$. By Assumption \ref{Ass}, these sequences converge in $C^\infty_{loc}(\R,M)$( up to taking a subsequence) to curves $\gamma^{(j)}$ which solve the following equations:
\beq
\gamma^{(0)} \quad \text{solves}\quad\frac{d}{ds}\gamma^{(0)}(s)+\nabla^g\left(\beta_\infty^+(s)h(\gamma^{(0)}(s))+(1-\beta_\infty^+(s))F(\gamma^{(0)}(s))\right)=0,
\eeq

\beq
\gamma^{(k+1)} \quad \text{solves}\quad\frac{d}{ds}\gamma^{(k+1)}(s)+\nabla^g\left(\beta_\infty^-(s)h(\gamma^{(k+1)}(s))+(1-\beta_\infty^-(s))F(\gamma^{(0)}(s))\right)=0.
\eeq
For $j=1,\ldots,k$,
\beq
\gamma^{(j)}\quad \text{solves}\quad\frac{d}{ds}\gamma^{(j)}(s)+\nabla^gh(\gamma^{(j)}(s))=0
\eeq
and the endpoints at $\pm\infty$ of those $\gamma{(j)}$ and also $\gamma^{(0)}(\infty)$ and $\gamma^{(k+1)}(-\infty)$ are critical points of $h$. Denote these critical points by $y_j^\pm:=\gamma^{(j)}(\pm\infty)$.
Furthermore, it follows from the definition of the sequences $\gamma_{n,j}$, that the values of $h$ at these critical points  are ordered as follows:
\beq
\label{order crit values}
h(y_0^+)\geq h(y_1^-)\geq h(y_1^+)\geq h(y_2^-)\geq h(y_2^+)\geq\ldots\geq h(y_k^+)\geq h(y_{k+1}^-).
\eeq
To prove now that we have at least $k+1$ distinct critical points of $h$, we will show that the inequalities between positive and negative ends of the trajectories $\gamma^{(j)}$ for $j=1,\ldots k$ are strict, i.e., that the trajectories $\gamma^{(j)}$ are non-constant.

Assume this is not the case. By definition of the moduli space $\M(n,x_1,\ldots,x_k,x_\ast^-,x_\ast^+)$, we have
\beq
\ev_n(\gamma_n)\in W^s(x_1,f_1)\times\ldots\times W^s(x_k,f_k).
\eeq
In particular, this implies $\gamma_n(nj)\in W^s(x_j,f_j)$ for all $n\in\N$. By definition of the evaluation map and the sequence $\gamma_{n,j}$, this shows that 
\beq
\gamma^{(j)}(0)=\lim_{n\to\infty} \gamma_n(nj)\in\overline{W^s(x_j,f_j)}.
\eeq
Since $\bar{f}_j$ is proper and bounded from below the closure of a stable manifold $W^s(x_j,f_j)$ is a union of stable manifolds of smaller dimension. As we assume that $\gamma^{(j)}$ is constant, thus a critical point of $h$, this contradicts assumption \eqref{clever_choice_of_f_i}. Thus none of the $\gamma^{(j)}$ is constant and all $x_i$ have index $\geq 1$ and the corresponding inequalities in \eqref{order crit values} are strict. This shows that indeed the point $\widetilde y_j:=y_j^-$ for $j=1,\ldots,k$ and $\widetilde y_{k+1}:=y_k^+$ are distinct critical points of $h$.

To see why the critical values of $\widetilde y_{1},\ldots, \widetilde y_{k+1}$ are in the interval $[-\|h-F\|,\,\|h-F\|]$, it suffices by \eqref{order crit values} to look at $h(y_0^+)$ and $h(y_{k+1}^-)$. We recall from Lemma \ref{bound F} that $|F(\gamma_n(s))|$ is bounded by $\|h-F\|$  for all $s\in\R$. This carries over to $\widetilde y_{1},\ldots, \widetilde y_{k+1}$ in the various limits. 

This finishes the proof of Theorem \ref{thm:main_intro}
\end{proof}

\section{Applications to Floer theory}
\label{sec:Floer}

\subsection{Fixed points of Hamiltonian diffeomorphisms}
\label{sec:Ham fixed points}

In this section, we apply Theorem \ref{thm:main_intro} to the action functional of Hamiltonian dynamics. The resulting bound on the number of critical points then yields a bound on the number of fixed points of Hamiltonian diffeomorphisms in terms of the cohomology of the underlying symplectic manifold. This reproduces previous proofs of the Arnold conjecture, see, e.g., \cite{Sch98}.

The proofs in the Floer case only differ in some details from the Morse theoretic proof above and we will point out the necessary changes in the proofs. The main differences are that different critical points can represent the same periodic orbit and that compactness can not only fail due to breaking trajectories as in the Morse case but also due to bubbling. Both issues can be excluded using the bounds on the Hamiltonian given by the condition on $h$ in Theorem \ref{thm:main_intro}.

We first describe the setting and identify the functionals which will take the role of the functions $F$ and $h$ from the Morse case.

Let $(W,\omega)$ be a closed, rational symplectic manifold, i.e. $\omega|_{\pi_2}=\lambda\Z$ for some positive $\lambda$. If $\omega|_{\pi_2}=0$, we set $\lambda=\infty$.
For this setting, we can now apply the Morse theoretic argument above to the Hamiltonian Floer action functional.
To do this, let $M=\widetilde{\Lambda W}$ be the covering space of  the space of contractible loops in $W$ and 
\beq
F=\A(\bar x)=-\int_{D^2}u^\ast\omega,
\eeq
where $\bar x$ is a loop $x$ together with an equivalence class $[u]$ of cappings. Two cappings are equivalent, if the two discs have the same symplectic area.

For this functional, the critical points are the constant loops with all different equivalence classes of cappings. These are Morse-Bott components since the symplectic form is non-degenerate. We define again 
\beq
Z=F^{-1}(0)=W
\eeq
to be the set of constant loops with constant cappings and identify this with the symplectic manifold $W$.
Furthermore, for this functional, we find $\S=\lambda$.

To define the function $h$, we define for a Hamiltonian $H\colon S^1\times W\to\R$ the action
\beq
\A_H(\bar x)=-\int_{D^2}u^\ast\omega+\int_{S^1}H(x)\,dt.
\eeq
Critical points of this functional are one-periodic orbits of the Hamiltonian flow of $H$, which can be identified with fixed points of the time one map $\varphi_H$. For those fixed points, we find the following analog of Theorem \ref{thm:main_intro} in the current setting:

\begin{Thm}[Existence of fixed points]
\label{thm:fixed points}
If $H$ has Hofer norm $\|H\|_H<\lambda$, then $H$ has at least $\mathrm{cuplength}(W) +1$ one-periodic orbits.
\end{Thm}

If $\lambda=\infty$, this lower bound on the number of one-periodic orbits holds for all Hamiltonians, as $W$ is compact and therefore all Hamiltonians have finite Hofer norm.

\begin{proof}
Using the Hamiltonian Floer equation instead of the gradient equation, we still define the moduli spaces $\M$ and $\M_{[0,R]}(Z)$ as in the Morse case. For $r=0$, the function $G_{r,s}$ is simply the unperturbed action $F$ and all Floer trajectories are constant.
The estimates in Lemma \ref{bound} also work in this case and give bounds in terms of the Hofer norm of the Hamiltonian $H$. Assumption \ref{Ass} also holds for all Floer theoretic situations as this provides compactness in $C^\infty_{loc}$ convergence of sequences of trajectories. The only possibility of this compactness to fail is bubbling. Now the definition of $\lambda$ shows that every bubble must have at least energy $\lambda$ and therefore, there is not enough energy for bubbling to occur by \eqref{energy}. This completes the proof of Assumption \ref{Ass} in this setting.

For the evaluation maps and the Morse functions $f_i$, however, we need to change the definitions.
Namely, for the definition of the evaluation maps, note that the elements in $\M_{[0,R]}(Z)$ can now be considered not only as paths in the loop space, but also as cylinders in the symplectic manifold $W$, i.e., as maps $\gamma\colon S^1\times\R\to W$, where $S^1=\R/\Z$.
Then we define the evaluation map by
\beq
\ev_r(\gamma)=\left( \gamma(0,r), \gamma(0,2r),\ldots, \gamma(0,kr)\right),
\eeq
as a map $\ev_r\colon\M_{[0,R]}(Z)\to W^k$ for $r\leq R$.

With this definition, it suffices now to define the Morse functions $f_i$ to be functions on $W$ and we can define the moduli spaces $\M(r,x_1,\ldots,x_k,x_\ast^-,x_\ast^+)$ as before. In particular, we point out that all evaluation maps already take values only in $W^k$ and therefore define the cup product in the cohomology of $W$.

\begin{Rmk}
We remark that we only use Morse cohomology to realize the cup product and do not rely on a product structure on Floer homology. 
\end{Rmk}

To show why the moduli space $\M_{[0,R]}(Z)$ is compact, we need to exclude breaking of Floer trajectories at critical points of $\A$ outside $Z$ and bubbling. Breaking is excluded by the same argument as in the Morse case; bubbling has been excluded above.
This completes the proof of compactness for $\M_{[0,R]}(Z)$ for the action functional of Hamiltonian dynamics.

From here on, the proof of Theorem \ref{thm:main_intro} goes through for the Floer case. As critical points of the functions $f_i$ represent cohomology classes of $W$, the proof gives the desired bound of at least $\mathrm{cuplength}(W) +1$ critical points of the actions functional with action value in the interval $(-\lambda,\,\lambda)$, since recapping changes the action by at least $\lambda$.

To see why these critical points actually are different periodic orbits of $H$, i.e., that no periodic orbit is found twice with different cappings, we use the energy bound discussed above. Again the bound on $\|H\|_H$ shows that a (broken) Floer trajectory cannot connect a periodic orbit to itself with a different capping.

This shows that we indeed have found $\mathrm{cuplength}(W) +1$ different one-periodic orbits.
\end{proof}

\subsection{Hamiltonian chords of a Lagrangian}

As for periodic orbits, we can also apply Theorem \ref{thm:main_intro} to find Hamiltonian chords of a Lagrangian submanifold to produce estimates similar to those in  \cite{Che98, Flo89, Hof88, Liu05}.

For this setting, let $(W,\omega)$ be a symplectic manifold and $L\subset W$ a Lagrangian.
Assume that $\omega|_{\pi_2(W,L)}$ is rational with rationality constant $\lambda\in (0,\infty]$.

Then we define $M$ as the set $P_0(W,L)$ of paths $x\colon [0,1]\to W$ with $x(0)$ and $x(1)$ in $L$, which are contractible to a path in $L$ and consider the action functional 
\beq
F=\A(\bar x)=-\int_{D^2}u^\ast\omega.
\eeq
Here, $\bar x$ is a path $x$ together with an equivalence class $[u]$ of discs bounded by the path $x$ and $L$. Analogously to above, two cappings are equivalent if the two discs have the same symplectic area.
For this functional, the critical points are the constant paths on $L$ with all different equivalence classes of cappings.
We define again
\beq
Z=F^{-1}(0)=L
\eeq
to be the set of critical points with constant cappings and identify this with the Lagrangian $L$. Again, this is Morse-Bott as the symplectic form is non-degenerate.
Furthermore, for this functional, we find $\S=\lambda$.

The functional $h$ is defined using a Hamiltonian $H\colon [0,1]\times W\to\R$.
Namely, we take
\beq
\A_H(\bar x)=-\int_{D^2}u^\ast\omega+\int_{[0,1]}H\,dt.
\eeq
Critical points of this functional are Hamiltonian chords of $L$ together with an equivalence class of cappings.

The above setting is again analogous to the situation of Theorem \ref{thm:main_intro} and we find the following

\begin{Thm}[Existence of Hamiltonian chords]
\label{thm:chords}
If $\|H\|<\lambda$, where $\|\cdot\|$ is the Hofer norm of $H$, then the Lagrangian $L$ has at least $\mathrm{cuplength}(L) +1$ Hamiltonian chords.
\end{Thm}

\begin{proof}
The proof follows again the proof of Theorem \ref{thm:main_intro}  with similar modifications as in the case of fixed points of Hamiltonian diffeomorphisms.

As the functional is analogous to the one for periodic orbits of $H$, we again find that $\|h-F\|$ is bounded by the Hofer norm of $H$. The evaluation map is defined by evaluating a path $x$ at time $0$ and by definition of $P_0(W,L)$, we have $x(0)\in L$. Therefore, we can choose the $f_i$ to be Morse functions on $L$ and find the cohomology operations to be the cup product in $L$ as again the evaluation maps take values in $L$. 
With these small modifications, the proof of Theorem \ref{thm:main_intro} applies also in this case and gives the desired lower bound for the number of critical points of $\A_H$ in terms of the cuplength of $L$.

We still need to show that these critical points correspond indeed to different Hamiltonian chords. Two different critical points of the action functional can give rise to the same Hamiltonian chord with a different capping, i.e., a different homotopy from the chord to a path in the Lagrangian $L$.

Similar to the case of periodic orbits, the requirement on the symplectic form and energy bounds in terms of $\|H\|$ guarantee that the broken trajectory does not have sufficient energy to connect a Hamiltonian chord to the same chord with a different capping. Namely, these two critical points of the action functional have action values that differ by at least $\lambda=\S$.
If two Hamiltonian chords are connected by a (broken) Floer trajectory $u$, the energy $E(u)$ is equal to the difference in action value and bounded above by $\|H\|<\S$. Therefore, we indeed find the desired number of different Hamiltonian chords.

\end{proof}

\subsection{Translated points}
\label{sec:translated points}

Next we apply Theorem \ref{thm:main_intro} to translated points, a notion introduced by Sandon \cite{San11}. Let $(\Sigma,\alpha)$ be a closed contact manifold. Let $\varphi:\Sigma\to\Sigma$ be a contactomorphism which is contact isotopic to the identity. We recall that a point $q\in\Sigma$ is a translated point with time-shift $\eta\in\R$
$$
\left
\{\begin{aligned}
\varphi(q)&=\theta^\eta(q)\\
\rho(q)&=1\;.
\end{aligned}
\right.
$$
We point out that the time-shift is not unique if $q$ lies on a closed Reeb orbit. The unperturbed action functional of Rabinowitz Floer homology has as critical points precisely the contractible Reeb orbits, see section \ref{sec:intro_translated}, and thus its spectral gap is 
\beq
\S=\text{minimal period of a contractible Reeb orbits}>0.
\eeq
Using the notation $\varphi^*\alpha=\rho\alpha$, where $\rho:\Sigma\to\R_{>0}$, a certain cut-off of the Hamiltonian diffeomorphisms $\phi:S\Sigma\to S\Sigma$, $\phi(q,r):=(\varphi(q),\frac{r}{\rho(q)})$ can be used to perturb the Rabinowitz action functional $F=\A$ to obtain $h=\A_\varphi$, see \cite{AFM13}, where it is proved in Lemma 3.5 that the critical points of $h$ correspond to translated points of $\varphi$. The perturbation is supported inside $\Sigma\times[e^{-\kappa(\varphi)},e^{\kappa(\varphi)}]\subset S\Sigma$
where
\beq
\kappa(\varphi):=\max_{t\in[0,1]}\left| \int_0^t\max_{x\in\Sigma}\frac{\dot{\rho}_{s}(x)}{\rho_{s}(x)^{2}}ds\right|\;.
\eeq
This implies that
\beq
\|h-F\|\leq e^{\kappa(\varphi)}\|H\|_{\H}
\eeq
where $H:S^1\times\Sigma\to\R$ is any contact Hamiltonian such that the induced contact isotopy $\psi_t$ satisfies $\varphi=\psi_1$.
\begin{Thm}\label{thm:translated_pts}
If $e^{\kappa(\varphi)}\|H\|_{\H}<\S$ then there are at least $\mathrm{cuplength}(\Sigma)+1$ many distinct translated points with time-shifts in the interval $[-e^{\kappa(\varphi)}\|H\|_{\H},e^{\kappa(\varphi)}\|H\|_{\H}]$.
\end{Thm}

\begin{proof}
Compared to the proof given in \cite{AM10}, which was the idea behind the proof of Theorem \ref{thm:main_intro} above, the only change is that we now work in the symplectization instead of assuming that the contact manifold $(\Sigma,\alpha)$ has a exact symplectic filling.
The only new point to the proof is the compactness of the moduli spaces $\M$ since the symplectization $S\Sigma$ has a negative end.
The assumption $e^{\kappa(\varphi)}\|H\|_{\H}<\S$ together with Stokes' theorem precisely implies that Theorem 3.9 in \cite{AFM13} is applicable. We conclude that there exists $\epsilon>0$ such that for each element in $\M$ the $\R_{>0}$ component of its image in the symplectization $S\Sigma=\Sigma\times\R_{>0}$ is bounded by $\epsilon$. Thus, all elements in $\M$ stay inside a compact subset of $S\Sigma$. Now as in \cite{AM10} and the Hamiltonian Floer case above, we define the evaluation maps by evaluating the trajectories at time $0$ and apply the proof of Theorem \ref{thm:main_intro} to obtain the lower bound for the number of translated points.
\end{proof}

Translated points are a special case of the notion of leafwise intersection points introduced by Moser in \cite{Mos78}. In \cite{AM10}, a lower bound on the number of leafwise intersections in terms of a relative cuplength for Liouville fillable contact manifolds has been proved. This result can be improved as follows. If we consider a Hamiltonian diffeomorphism $\phi$ with support inside $\Sigma\times[e^{-\kappa},e^{\kappa}]\subset S\Sigma$ for some $\kappa>0$, then if $e^{\kappa(\varphi)}\|H\|_{\H}<\S$ there exist at least $\mathrm{cuplength}(\Sigma)+1$ leafwise intersections. Now $H:S^1\times S\Sigma\to \R$ is any Hamiltonian function generating $\phi$. If the contact manifold $\Sigma$ is Liouville fillable then $\phi$ extends by the identity to the filling. We point out that in this situation $\mathrm{cuplength}(\Sigma)$ is always at least as big the relative cuplength used in \cite{AM10} and thus our result here improves the bound given in \cite{AM10}.

\subsection{Solutions to perturbed Dirac-type equations}
\label{sec:hyperkahler}
In this section, we apply Theorem \ref{thm:main_intro} to the hyperk\"ahler Floer homology developed by Hohloch, Noetzel and Salamon in
\cite{HNS09} and reprove the cuplength estimate by Ginzburg and the second author in
\cite{GH12}.

Recall that we take the "time"-manifold $X$ to be either $T^3$ or $S^3$ equipped with a volume form $\mu$ and a special choice of a global frame.
If $X=T^3$, we choose the global frame $v_1, v_2, v_3$ on $X$ by $\partial_{t_i}$ on $T^3$ for angular coordinates $t_1, t_2,t_3$ and the volume form the be $\mu=dt_1\wedge dt_2\wedge dt_3$.
If $X=S^3$, we identify $S^3$ with the unit quaternions and define $v_1(x)=ix$, $v_2(x)=jx$ and $v_3(x)=kx$ and choose the volume form $\mu$ to be the (probability) Haar measure on $S^3$.

Let $Y$ be a compact, flat hyperk\"ahler manifold with almost complex structures $I, J, K$. This implies that $Y$ is some compact quotient of a hyperk\"ahler vector space, i.e., a torus or a quotient of a torus by a finite group. On $Y$, we define symplectic forms $\omega_i$ by choosing a Riemannian metric which is compatible with any linear combination of the almost complex structures $I, J$ and $K$ and setting $\omega_1(\cdot,\cdot)=\left<\cdot,I\cdot\right>$, $\omega_2(\cdot,\cdot)=\left<\cdot,J\cdot\right>$ and $\omega_3(\cdot,\cdot)=\left<\cdot,K\cdot\right>$

In this setting, we consider the manifold $M$ to be the space of smooth null-homotopic maps from $f\colon X\to Y$ and define the action functional
\bea
\A(f)=-\sum_{l=1}^3\int_{[0,1]\times X} \hat f^\ast\omega_l\wedge i_{v_l}\mu,
\eea
where $\hat f\colon [0,1]\times X\to Y$ is a homotopy from $f$ to a constant map. As the covering space of $Y$ is contractible, this functional is independent of the choice of $\hat f$ and only depends on the map $f$.
This action functional is the one defined in \cite{GH12} and agrees with the formulas given in \cite{HNS09}.
The differential of the action functional $\A$ at $f$ is
\bea
(d\A)_f(w)=\sum_{l=1}^3\int_M\omega_l(L_{v_l}f,w)\mu
\eea
and the $L^2$-gradient of $\A$ at $f$ is given by $\partialslash f:=IL_{v_1}f+JL_{v_2}f+KL_{v_3}f$.
 Critical points of $\A$ are solutions to the Dirac-type equation 
 \bea
 \partialslash f=0.
\eea
The functional $\A$ takes the role of the function $F$ from the Morse case.
All critical points of $F$ are constant, i.e., we again identify the critical submanifold $Z$ with the manifold $Y$ and find that the spectral gap is $\S=\infty$. The functional $\A$ is Morse-Bott along $Z$, see \cite[Lemmas 2.5 \& 3.7]{HNS09}.

Just as in the case of classical Hamiltonian dynamics described above, we consider a Hamiltonian perturbation for a Hamiltonian $H\colon X\times Y\to \R$. The function $h$ is then given by the functional
\bea
h(f):=\A_H(f)=-\sum_{l=1}^3\int_{[0,1]\times X} F^\ast\omega_l\wedge i_{v_l}\mu-\int_XH(f)\mu.
\eea
For this functional, critical points are solutions to the equation
\bea
\partialslash f=\nabla H(f).
\eea
As in the case of classical Hamiltonian dynamics discussed above, Theorem \ref{thm:main_intro} implies the following
\begin{Thm}
\label{thm:cuplength-Dirac}
For all Hamiltonians $H$, the number of critical points of $\A_H$ is bounded below by $\mathrm{cuplength}(Y)+1$.
\end{Thm}

\begin{proof}
In this setting and with these choices of global frames, Assumption \ref{Ass} is established in \cite{HNS09}, where also the hyperk\"ahler Floer homology is defined. 
To define the evaluation maps, we evaluate the trajectories at $x=(0,0,0)$ if $X=T^3$ or, if $X=S^3$ is identified with the unit quaternions, at $x=1$. Then the evaluation maps take values in $Y$ and we can choose the Morse functions $f_i$ to be functions on $Y$. As in the case of Hamiltonian Floer homology discussed above, the proof of Theorem \ref{thm:main_intro} gives the lower bound for the number of critical points of $\A_H$ to be $\mathrm{cuplength}(Y)+1$ for all Hamiltonians $H$.
\end{proof}

%

\end{document}